\documentclass[journal,letterpaper,onecolumn,twoside,nofonttune]{IEEEtran}

\usepackage[utf8]{inputenc}
\usepackage[T1]{fontenc}
\usepackage{url}
\usepackage{ifthen}
\usepackage{cite}
\usepackage[cmex10]{amsmath}
\usepackage{balance}

\usepackage{amsmath}
\usepackage{amsthm}
\usepackage{amsfonts}
\usepackage{amssymb}
\usepackage{mathtools}
\usepackage{color}
\usepackage{verbatim}
\usepackage[normalem]{ulem}
\usepackage{enumitem}
\usepackage{hyperref}
\usepackage[ruled,vlined]{algorithm2e}
\usepackage{algpseudocode}

\usepackage{caption}
\newcommand{\N}{\mathbb{N}}

\newcommand{\R}{\mathbb{R}}

\newcommand{\E}{\mathbb{E}}

\newcommand{\rpm}{\raisebox{.2ex}{$\scriptstyle\pm$}}

\newcommand{\vect}[1]{\boldsymbol{#1}}

\newcommand{\Scal}{\mathcal{S}}

\newcommand{\eb}{\vect{e}}

\newcommand{\xb}{\vect{x}}
\newcommand{\xt}{\tilde{x}}

\newcommand{\Ab}{\vect{A}}

\newcommand{\Bb}{\vect{B}}

\newcommand{\Bbbar}{\bar{\vect{B}}}
\newcommand{\Bbh}{\hat{\vect{B}}}
\newcommand{\Bbt}{\tilde{\vect{B}}}
\newcommand{\Bbti}{\tilde{\vect{B}}_{\mathrm{iso}}}

\newcommand{\Db}{\vect{D}}

\newcommand{\Et}{\tilde{E}}

\newcommand{\Gt}{\tilde{G}}

\newcommand{\Lb}{\vect{L}}
\newcommand{\Lbt}{\tilde{\vect{L}}}

\newcommand{\Sb}{\vect{S}}

\newcommand{\wt}{\tilde{w}}

\newcommand{\Ib}{\vect{I}}

\newcommand{\Pib}{\bold{\Pi}}

\newcommand{\chit}{\tilde{\chi}}

\newcommand{\crmm}{CR\mathrm{-MM}}

\definecolor{LightBlue}{rgb}{0, 0.7, .7}
\definecolor{Green1}{rgb}{0.0, 0.5, 0.0}
\definecolor{green}{rgb}{0.0, 0.42, 0.24}
\definecolor{byzantine}{rgb}{0.74, 0.2, 0.64}

\hypersetup{colorlinks,linkcolor={Green1},citecolor={blue},urlcolor={red}}

\newtheorem{Thm}{Theorem}

\newtheorem{Prop}{Proposition}
\newtheorem{Def}{Definition}

\newcommand{\ind}{\text{\color{white}.$\quad$}}
\DeclareMathOperator*{\argmin}{arg\,min}

\begin{document}

\title{{\huge Graph Sparsification by Approximate Matrix Multiplication}}

\author{Neophytos Charalambides, Alfred O. Hero III\\
EECS Department, University of Michigan, Ann Arbor, MI 48109\\
  Email: \texttt{\textbf{neochara@umich.edu}}, \texttt{\textbf{hero@umich.edu}}
\thanks{We thank Greg Bodwin for helpful discussions and suggestions. This work was partially supported by the US Department of Energy grant DE-NA0003921, and the Army Research Office grant W911NF1910269.}}

\maketitle


\begin{abstract}
Graphs arising in statistical problems, signal processing, large networks, combinatorial optimization, and data analysis are often dense, which causes both computational and storage bottlenecks. One way of \textit{sparsifying} a \textit{weighted} graph, while sharing the same vertices as the original graph but reducing the number of edges, is through \textit{spectral sparsification}. We study this problem through the perspective of RandNLA. Specifically, we utilize randomized matrix multiplication to give a clean and simple analysis of how sampling according to edge weights gives a spectral approximation to graph Laplacians, without requiring spectral information. Through the $\crmm$ algorithm, we attain a simple and computationally efficient sparsifier whose resulting Laplacian estimate is unbiased and of minimum variance. Furthermore, we define a new notion of \textit{additive spectral sparsifiers}, which has not been considered in the literature.
\end{abstract}

\begin{IEEEkeywords}
Large graphs, Laplacians, spectral sparsification, numerical linear algebra, graph approximations, random sampling.
\end{IEEEkeywords}

\section{Introduction and Related Work}
\label{intro}

Large graphs, networks and their associated Laplacian are prevalent in many applications and domains of modern signal processing, statistics and engineering, \textit{e.g.} spectral clustering \cite{Von07}, community detection \cite{DCT20} and graph learning \cite{JYKS21}. Their size makes them hard to store and process, which is why it is preferred to instead work with a good approximation or sketch of the graph. Algorithms for approximating large graphs have been developed through the study of \textit{spectral graph theory}, which deals with the eigenvalues and eigenvectors of matrices naturally associated with graphs. A standard approach is by sampling edges or vertices of these graphs, with judiciously chosen sampling distributions.

Our main contribution, is bridging a connection between \textit{randomized numerical linear algebra} (RandNLA) and approximate \textit{matrix multiplication} (MM), with Laplacian \textit{spectral sparsifiers} of \textit{weighted} graphs $G=(V,E,w)$. The resulting algorithm is intuitive and simple, and has been considered in independent works. Our analysis though is more straightforward and shorter than other analyses considering the same and similar sparsifiers, \textit{e.g.} \cite{SYT16}. Lastly, we introduce an alternative measure for spectral sparsifiers, which captures additive approximation errors.

Spectral sparsifiers are of importance, as they preserve eigenvector centrality \cite{ST04}, cuts in a graph \cite{BK96}, flows in networks modeled by graphs \cite{GT88}, and maintain the structure of the original graph. By viewing the Laplacian $\Lb$ of $G$ as the outer product of its boundary matrix $\Bb\in\R_{\geqslant0}^{E\times V}$, we use $CR$ matrix multiplication ($\crmm$) to approximate the Laplacian $\Lbt\approx\Lb$. This turns out be equivalent to sampling and re-weighting edges from $G$, with sampling probabilities proportional to the edges' weights. The resulting Laplacian $\Lbt$ is an unbiased estimate of minimum variance, and represents the sketched graph $\Gt=(V,\Et,\wt)$. Unlike most other spectral sparsifiers whose guarantees depend entirely on the number of vertices $n$, ours depends on the edge weights $w$.

What we present also draws connections between sampling according to Frobenius norm of vectors, and \textit{leverage scores}; which has been extensively studied in the context of linear systems and $\ell_2$-subspace embeddings \cite{DMMS11,DMMW12}. Sparsifying Laplacians through sampling is the appropriate intermediate application, between MM and subspace embeddings.

\subsection{Related Work}

The main idea behind the sparsifier we study is simple and intuitive. By using a primitive which has extensively been studied; approximate multiplication, as a surrogate to analysing the proposed spectral sparsifier, we present a simple analysis which yields more concise statements regarding the resulting sparsifier, compared to related work \cite{SYT16}; which considers Gaussian smoothing. Our guarantees differ from previous works, and we draw connections to RandNLA.

Along similar lines, connections between \textit{effective resistances} and leverage scores have been previously established \cite{DM10,SS11}. Spectral sparsification has also been used in linear algebra to obtain deterministic and randomized algorithms for low-rank matrix approximations \cite{BDM14,Cam14}. In this work, we obtain results in the converse direction.

The state-of-the-art approach to spectral sparsification is to sample edges according to effective resistances \cite{Spi10,SS11}. This approach leads to a nearly-linear time algorithm that produces high-quality sparsifiers of weighted graphs. A drawback of this approach is the computational complexity of determining the resistances, which requires either a spectral decomposition of $\Lb$, or directly computing $\Lb^{\dagger}$. In what we propose, the sampling distribution is already known through $w$, and the sampling can be done pass-efficiently only inquiring an additional $O(1)$ additional storage space \cite[Algorithm 1]{Mah16}. This makes our method algorithmically superior to sampling according to effective resistances, as computing them requires $O(|E|\cdot|V|^2)$ operations.

Furthermore, through leverage scores, sampling according to the effective resistances relates to the notion of an \textit{$\ell_2$-subspace embedding}. As contrasted to the objective of \cite{SS11}, we use approximate multiplication; to obtain minimum variance unbiased estimators.

\subsection{Preliminaries}
\label{prelim_subsec}

Recall that the \textit{\textbf{Laplacian}} of $G=(V,E,w)$ a weighted undirected graph with $|V|=n$, $|E|=m$, weights $w_{i,j}$ for each edge $(i,j)\in E$ is
\begin{equation}
\label{lapl_eq}
  \Lb_{ij} = \begin{cases} \sum_{(i,\ell)\in E}w_{i,j} \ \ \text{ if } i=j \\ -w_{i,j} \qquad \qquad \text{if } i\neq j \end{cases} \ \ \text{ for }  i,j\in V.
\end{equation}
Equivalently, it is expressed as $\Lb=\Db-\Ab$, for $\Db,\Ab\in\N_0^{n\times n}$ respectively the degree and adjacency matrices of $G$. This can also be expressed as the Gram matrix of the \textbf{\textit{boundary matrix}}\footnote{The transpose of the boundary matrix of $G$, is also known as the \textit{incidence matrix} of $G$.} $\Bb\in\R^{E\times V}$. Once we determine an arbitrary positive orientation $(i,j)$ of the edges in $E$, the boundary matrix associated with the orientation is defined as
\begin{equation*}
\label{bound_eq}
  \Bb_{(i,j),v} = \begin{cases}-\sqrt{w_{i,j}} \ \ \ \text{ if } v=i \\ \sqrt{w_{i,j}} \ \ \ \ \ \ \text{ if } v=j\\ 0 \qquad \qquad \text{o.w.} \end{cases} \ \ \text{ for }  (i,j)\in E \text{ and } v\in V.
\end{equation*}
For an edge $e=(u,v)$, the orientation is represented in the \textbf{\textit{incidence vector}} $\chi_e=\eb_u-\eb_v$; for $\eb_i\in\R^V$ the standard basis vectors. We define the \textbf{\textit{weighted incidence vector}} as $\chit_e=\sqrt{w_e}\cdot\chi_e$. The Laplacian of $G$ is then
\begin{equation}
\label{lapl_bound}  
  \Lb = \Bb^T\Bb = \sum_{e\in E}\chit_e\chit_e^T = \sum_{e\in E}w_e\cdot\chi_e\chi_e^T \in\R^{V\times V}.
\end{equation}

\subsection{Approximate Matrix Multiplication}
\label{CR_subsec}

Consider the two matrices $A\in\R^{L\times N}$ and $B\in\R^{N\times M}$, for which we want to approximate the product $AB$. It is known that the product may be approximated by sampling with replacement (s.w.r.) columns of $A$ and rows of $B$, where the row-column sampling probabilities are proportional to their Euclidean norms. That is, we sample with replacement $r$ pairs $(A^{(i)},B_{(i)})$ for $i\in\N_N\coloneqq\{1,\cdots,N\}$ and $r<N$ ($A^{(i)}$=$i^{th}$ column of $A$, and $B_{(i)}$=$i^{th}$ row of $B$), with probability
\begin{equation}
\label{sampl_pro}
  p_i=\frac{\|A^{(i)}\|_2\cdot\|B_{(i)}\|_2}{\sum_{l=1}^N\|A^{(l)}\|_2\cdot\|B_{(l)}\|_2}
\end{equation}
and sum a rescaling of the samples' outer-products:
\begin{equation}
\label{CR_eq}
  AB \approx \frac{1}{r}\cdot \left(\sum_{j\in\Scal}\frac{1}{p_j}A^{(j)}B_{(j)}\right) = \sum_{j\in\Scal}\frac{A^{(j)}}{\sqrt{rp_j}}\cdot\frac{B_{(j)}}{\sqrt{rp_j}} \eqqcolon Y
\end{equation}
where $\Scal$ is the multiset consisting of the indices (possibly repeated) of the sampled pairs, hence $|\Scal|=r$. We denote the corresponding ``compressed versions'' of the input matrices by $C\in\R^{L\times r}$ and $R\in\R^{r\times M}$ respectively. This approximation satisfies $\|AB-CR\|_F=O(\|A\|_F\|B\|_F/\sqrt{r})$. Further details on this algorithm may be found in \cite{DK01,DKM06a,DKM06b,Woo14,Mah16}. We have the following known results for the $\crmm$ algorithm.

\begin{Thm}[Section 3.2 \cite{Mah16}]
\label{thm_CR}
The estimator $Y=CR$ from \eqref{CR_eq} is unbiased, while the sampling probabilities $\{p_i\}_{i=1}^N$ minimize the variance, \textit{i.e.}
\begin{equation}
\label{eq_min_var_CR}
  \{p_i\}_{i=1}^N = \argmin_{\substack{\sum_{i=1}^Np_i=1}} \Big\{\mathrm{Var}(Y)=\E\left[\|AB-CR\|_F^2\right]\Big\}
\end{equation}
and it is an $\epsilon$-multiplicative error approximation of the matrix product, with high probability. Specifically, for $\delta\geqslant0$ and $r\geqslant \frac{1}{\delta^2\epsilon^2}$ the number of sampling trials which take place
\begin{equation}
\label{CR_mult_error}
  \Pr\big[\|AB-CR\|_F \leqslant\epsilon\cdot\|A\|_F\|B\|_F\big]\geqslant1-\delta
\end{equation}
for any $\epsilon>0$.
\end{Thm}

\begin{Thm}[Theorem 8\cite{Mah16}]
\label{sp_bd_thm}
Let $A\in\R^{L\times N}$ with $\sigma_{\max}(A)=\|A\|_2\leqslant1$, and approximate the product $Y\approx AA^T$ using $\crmm$. Let $\epsilon\in(0,1)$ be an accuracy parameter, and assume that $\|A\|_F^2\geqslant1/24$. If
\begin{equation*}
\label{samples_AAt}
  r\geqslant\frac{96\|A\|_F^2}{\epsilon^2}\ln\left(\frac{96\|A\|_F^2}{\epsilon^2\sqrt{\delta}}\right)\geqslant\frac{4}{\epsilon^2}\ln\left(\frac{4}{\epsilon^2\sqrt{\delta}}\right)
\end{equation*}
for $r\leqslant N$, then
\begin{equation}
\label{bd_AAt}
  \Pr\left[\|AA^T-Y\|_2\leqslant\epsilon\right] \geqslant 1-\delta\ .
\end{equation}
\end{Thm}

Below, we provide the pseudocode of the $\crmm$ algorithm.

\begin{algorithm}[h]
\label{CR_mult}
\SetAlgoLined
 \KwIn{Matrices $A\in\R^{L\times N}$ and $B\in\R^{N\times M}$}
 \KwOut{Approximate product $Y\approx AB$}
 \textbf{Determine}: Distribution $\{p_i\}_{i=1}^N$, according to \eqref{sampl_pro}\\
 \textbf{Initialize}: $Y=\bold{0}_{L\times M}$\\
 \For{i $\gets$ 1 to r}
   {
     sample $j\in \N_N$ with replacement, according to $\{p_i\}_{i=1}^N$\\
     $Y\gets Y + \frac{1}{rp_j}\cdot A^{(j)}B_{(j)}$
   }
 \caption{$CR$ matrix multiplication}
\end{algorithm}

\section{Spectral Sparsification}
\label{sp_spars_sec}

First, recall that an \textbf{\textit{$\varepsilon$-spectral sparsifier}} for $\varepsilon\in(0,1)$ of $G$ with Laplacian $\Lb$, is a sketched graph $\Gt$ whose Laplacian $\Lbt$ satisfies
\begin{equation}
\label{sp_spars_id}
(1-\varepsilon)\xb^T\Lbt\xb \leqslant \xb^T\Lb\xb \leqslant (1+\varepsilon)\xb^T\Lbt\xb \quad \iff \quad (1-\varepsilon)\|\Bbt\xb\|_2^2 \leqslant \|\Bb\xb\|_2^2 \leqslant (1+\varepsilon)\|\Bbt\xb\|_2^2
\end{equation}
for all $\xb\in\R^n$. This implies that the approximated graph $\Gt$ preserves the total weight of any cut between the factors of $1\rpm\varepsilon$, hence also allowing a good approximation to its max-flow. A natural definition to consider, is that of when the approximation error is \textit{additive}.

\begin{Def}
\label{add_sp_def}
An \textbf{additive $\varepsilon$-sparsifier} of $G$ with Laplacian $\Lb$, is a sketched graph $\Gt$ whose Laplacian $\Lbt$ satisfies

$$ \xb^T\big(\Lbt-\varepsilon\cdot\Ib_n\big)\xb \leqslant \xb^T\Lb\xb \leqslant \xb^T\big(\Lbt+\varepsilon\cdot\Ib_n\big)\xb \quad \iff \quad \left|\xb^T(\Lb-\Lbt)\xb\right| \leqslant \varepsilon\cdot\|\xb\|_2^2\ , $$
for all $\xb\in\R^n$.
\end{Def}

We distinguish between the two types of sparsifiers, by referring to those satisfying \eqref{sp_spars_id} as \textit{\textbf{multiplicative}}. It is worth pointing out that row/column sampling algorithms whose approximations are in terms of the Frobenius norm; \textit{e.g.} \eqref{CR_mult_error}, naturally yield additive sparsifiers, while those which are in terms of the Euclidean norm; \textit{e.g.} \eqref{bd_AAt}, admit multiplicative sparsifiers.

\subsection{Spectral Sparsifier from $\crmm$}

We propose approximating $\Lb$ by using the $\crmm$ algorithm on $\Bb^T\Bb$. Let $W=\|\Bb\|_F^2/2=\sum_{e'\in E}w_{e'}$. The resulting sampling probability of $e\in E$ according to \eqref{sampl_pro}, is
\begin{equation}
\label{sampl_pro_w}
  p_e\propto\|\Bb_{(e)}\|_2^2 = 2w_e \quad \implies \quad p_e=w_e/W\ .
\end{equation}
Thus, we are sampling edges proportionally to their weights. The resulting procedure is presented in Algorithm \ref{CR_spars}, where at each iteration we have a rank-1 update. We carry out a total of $r$ sampling trials and rescale the updates, to reduce the variance of the estimator. Moreover, for $\Pib=\text{diag}(w_e/W)$, let $x_e=\sqrt{\Pib}\chi_e$. Then $p_e=\|x_e\|_2^2$.

In simple words, we carry out $r$ sampling trials with replacement on $E$, and each time $e'$ is sampled, its new weight is increased by $\frac{W}{r}$. Furthermore, we note that the sampling procedure results in a diagonal sketching matrix $\Sb$, where $\Sb_{e,e}=\frac{\#\ e\ {\text is} \ {\text sampled}}{rp_e}$. Hence $\Lbt=\Bb^T\Sb\Bb$ and $\Bbt=\sqrt{\Sb}\Bb$.

\begin{algorithm}[h]
\SetAlgoLined
 \KwIn{A weighted simple undirected graph $G=(V,E,w)$, number of sampling trials $r$}
 \KwOut{Laplacian $\Lbt$, of sparsified $\Gt=(V,\Et,\wt)$}
 \textbf{Determine}: Boundary matrix $\Bb\in\R^{E\times V}$ of $G$, distribution $\{p_e=w_e/W\}_{e\in E}$\\
 \textbf{Initialize}: $\Lbt=\bold{0}_{V\times V}$\\
 \For{i $\gets$ 1 to r}
   {
     sample w.r. $e'\in E$, according to $\{p_e\}_{e\in E}$\\
     $\Lbt\gets\Lbt + \frac{W}{rw_{e'}}\cdot\chit_{e'}\chit_{e'}^T = \Lbt + \frac{W}{r}\cdot\chi_{e'}\chi_{e'}^T$
   }
 \caption{$CR$ spectral sparsifier}
\label{CR_spars}
\end{algorithm}

\begin{Prop}
\label{prop_add_ss}
Given a weighted simple undirected graph $G=(V,E,w)$, Algorithm \ref{CR_spars} produces an additive $\varepsilon$-spectral sparsifier of minimum variance; for $\varepsilon=2W\epsilon$ and $\epsilon$ the $CR$ accuracy parameter, with probability $1-\delta$ and $r\geqslant\frac{1}{\delta^2\epsilon^2}$.
\end{Prop}

\begin{proof}
From \eqref{CR_mult_error}, for $\Delta\coloneqq \Lb-\Lbt\succeq 0$ we have w.h.p. $\|\Delta\|_F = \|\Lb-\Lbt\|_F\leqslant\epsilon\|\Bb\|_F^2=2W\epsilon$, and in turn:
\begin{align*}
  \xb^T(\Lb-\Lbt)\xb &\overset{\sharp}{=} \xb^T\Delta\xb \\
  &= \|\xb^T\Delta\xb\|_F \\
  &\leqslant \|\Delta\|_F\cdot\|\xb\|_2^2 \\
  &= \big(2W\epsilon\big)\cdot\|\xb\|_2^2 ,
\end{align*}
which implies that
$$\xb^T\Lb\xb \leqslant \xb^T\left(\Lbt+\Ib_n\cdot\big(2W\epsilon\big)\right)\xb .$$
In the case where $\Delta\preceq0$, continuing from $\sharp$ we have
$$ -\xb^T\Delta\xb = \|\xb^T\Delta\xb\|_F \leqslant  \big(2W\epsilon\big)\cdot\|\xb\|_2^2 \quad \implies \quad \xb^T\left(\Lbt-\Ib_n\cdot\big(2W\epsilon\big)\right)\xb \leqslant \xb^T\Lb\xb. $$
All in all we have
\begin{equation}
\label{add_spars_PSD}
 \xb^T\left(\Lbt-\Ib_n\big(2W\epsilon\big)\right)\xb \leqslant \xb^T\Lb\xb \leqslant \xb^T\left(\Lbt+\Ib_n\big(2W\epsilon\big)\right)\xb \quad
  \iff \quad \left|\xb^T(\Lb-\Lbt)\xb\right| \leqslant \big(2W\epsilon\big)\cdot\|\xb\|_2^2\ ,
\end{equation}
which is an additive $\varepsilon$-spectral sparsifier; for $\varepsilon=\big(2W\epsilon\big)$.

By Theorem \ref{thm_CR}, the resulting estimator is of minimum variance. If $r\geqslant \frac{1}{\delta^2\epsilon^2}$ sampling trials are carried out, by \eqref{CR_mult_error} it follows that we attain such a sparsifier with probability at least $1-\delta$.
\end{proof}

\subsection{Multiplicative Spectral Sparsifier}
\label{mult_sp_spars_sub}

The case where $A=B^T$ in the $\crmm$ algorithm has also been studied as a special case, as it appears in numerous applications. This restriction allows us to use statements from random matrix theory \cite{Oli10}, to get stronger spectral norm bounds, \textit{e.g.} Theorem \ref{sp_bd_thm} \cite[Theorem 4]{DMMS11}, \cite[Theorem 8]{Mah16}.

We will use Theorem \ref{sp_bd_thm} to show that Algorithm \ref{CR_spars} is also a multiplicative spectral sparsifier. First, we recall an equivalent definition of a multiplicative $\varepsilon$-spectral sparsifier, based on spectral norm.

\begin{Def}
\label{def_SS_sp_norm}
For a weighted graph with Laplacian $\Lb$ and $\varepsilon>0$, a sketched graph $\Gt$ of $G$ with Laplacian $\Lbt$ and \textbf{isotropic boundary matrix} $\Bbti\coloneqq\Bbt\Lb^{-1/2}$ for $\Lb^{-1/2}\coloneqq\sqrt{\Lb^{\dagger}}$, is a \textbf{multiplicative $\varepsilon$-spectral sparsifier} if
\begin{equation}
\label{isot_cond}  
  \|\Ib_n-\Bbti^T\Bbti\|_2 = \|\Lb^{-T/2}(\Lb-\Lbt)\Lb^{-1/2}\|_2 \leqslant\varepsilon\ .
\end{equation}
\end{Def}

\begin{Prop}
\label{prop_mult_ss}
Let $G=(V,E,w)$ be a weighted simple undirected graph with $W=\sum_{e'\in E}w_{e'}\geqslant\sigma_{\max}^2(\Bb)/48$, and $\epsilon\in(0,1)$ an accuracy parameter.\footnote{$\lambda_{\max}(\Lb)=\sigma_{\max}(\Lb)=\sigma_{\max}^2(\Bb)$} Algorithm \ref{CR_spars} produces a multiplicative $\varepsilon$-spectral sparsifier $\Gt$ for $\varepsilon=\kappa_2(\Lb)\cdot\epsilon$ with high probability, for $r$ sufficiently large.\footnote{The \textit{\textbf{condition number}} of $\Lb$ is denoted by $\kappa_2(\Lb)=\|\Lb\|_2\|\Lb^{\dagger}\|_2=\sigma_{\max}(\Lb)/\sigma_{\min}(\Lb)$ \cite{Spi10,Vis13,Kes16}. Since the smallest singular of $\Lb$ for $G$ connected is 0, by $\sigma_{\min}(\Lb)$ we denote the second smallest singular, which is equal to $1/\|\Lb^{\dagger}\|_2$. Also note that $\Lb^{-T/2}=\Lb^{-1/2}$.}
\end{Prop}

\begin{proof}
Denote the sketch of Algorithm \ref{CR_spars} by $\Bb^T\Bb\approx\Bbt^T\Bbt$, and define $\Bbbar\coloneqq\Bb/\sigma_{\max}(\Bb)$; $\Bbh\coloneqq\Bbt/\sigma_{\max}(\Bb)$. Let $\Bbbar^T\gets A$ in Theorem \ref{sp_bd_thm}, thus $\Bbbar^T\Bbbar\approx\Bbh^T\Bbh$. The first condition of Theorem \ref{sp_bd_thm} is met, as $\|\Bbbar\|_2=\|\Bb\|_2/\sigma_{\max}(\Bb)=1$. Since $\|\Bbbar\|_F^2=\|\Bb\|_F^2/\sigma_{\max}^2(\Bb)=2W/\sigma_{\max}^2(\Bb)$, by our assumption on $W$ it follows that $\|\Bbbar\|_F^2=\frac{2W}{\sigma_{\max}^2(\Bb)}\geqslant\frac{2\sigma_{\max}^2(\Bb)}{48\sigma_{\max}^2(\Bb)}=1/24$. Hence, the condition $\|\Bbbar\|_F^2\geqslant1/24$ is also met.

Let $\theta=\sigma_{\max}(\Lb)\epsilon=\sigma_{\max}^2(\Bb)\epsilon$. From \eqref{bd_AAt} it follows that:
\begin{align*}
  \Pr\left[\|\Lb-\Lbt\|_2\leqslant\sigma_{\max}(\Lb)\epsilon\right] &= \Pr\left[\frac{\|\Bb^T\Bb-\Bbt^T\Bbt\|_2}{\sigma_{\max}^2(\Bb)}\leqslant\epsilon\right]\\
  &= \Pr\left[\|\Bbbar^T\Bbbar-\Bbh^T\Bbh\|_2\leqslant\epsilon\right]\\
  &\leqslant 1-\delta\ .
\end{align*}
We now appropriately apply $\Lb^{-1/2}$, in order to invoke \eqref{isot_cond}:
\begin{align*}
  \|\Ib_n-\Bbti^T\Bbti\|_2 &= \|\Ib_n-\Lb^{-1/2}\cdot(\Bbt^T\Bbt)\cdot\Lb^{-1/2}\|_2\\
  &= \|\Ib_n-\Lb^{-1/2}\cdot\Lbt\cdot\Lb^{-1/2}\|_2\\
  &= \|\Lb^{-1/2}(\Lb-\Lbt)\Lb^{-1/2}\|_2\\
  &\leqslant \theta\cdot\|\Lb^{-1/2}\|_2^2\\
  &= \frac{\theta}{\sigma_{\min}(\Lb)}\\
  &= \kappa_2(\Lb)\cdot\epsilon\ .
\end{align*}
Therefore
$$ \Pr\left[\|\Ib_n-\Bbti^T\Bbti\|_2\leqslant \kappa_2(\Lb)\cdot\epsilon\right] \geqslant 1-\delta $$
for $r\geqslant 6\gamma_{\epsilon,\Bb}^2\ln\left(\gamma_{\epsilon,\Bb}^2/\sqrt{\delta}\right)$,
where $\gamma_{\epsilon,\Bb}=\frac{8W}{\epsilon\cdot\sigma_{\max}(\Bb)}$ and $\delta\in(0,1]$. This completes the proof.
\end{proof}

We note that since the objective here is to sparsify the graph, and since we do so by s.w.r., the condition $r\leqslant N$ assumed in Theorem \ref{sp_bd_thm} can be violated, as we will get heavier resulting edges for unstructured graphs, rather than more edges. All guarantees will still hold true.

\subsection{Comparison to the Effective Resistances Approach}

Let $\xt_e=\Lb^{-1/2}\chi_e$, for each $e\in E$. Then, the effective resistances are defined as $r_e=\|\xt_e\|_2^2$. It is therefore clear that the only difference between the proposed algorithm and that of sparsifying through effective resistances, is that the former is rescaled according to $\Pib$ rather than $\Lb^{\dagger}$. The main benefit in our approach, is that the sampling distribution can be determined directly through $w$. We note also that $\{r_e\}_{e\in E}$ can be \textit{approximated} in nearly-linear time \cite{SS11}.

The analysis of the proposed random sampling algorithm invokes Theorem \ref{sp_bd_thm}, whose proof relies on a Chernoff bound on sums of Hermitian matrices \cite{Oli10}. Use of this bound is new in random sampling for Laplacian sparsification, and specifically applies to our proposed spectral method using sampling with replacement. This is to be compared with the use of other conventional Chernoff bounds \cite{Tro12} and concentration of measure \cite{Rud99} approaches. Intriguingly, unlike \cite{SS11}; our approach does not require spectral information of $\Lb$.

The benefit of using the bound of \cite{Oli10}, is that it can be applied to directly approximate the \textit{intersection} of two different graphs on $V$. Specifically, we use $\crmm$ to approximate $\Lb_{1,2}=\Bb_1^T\Bb_2$, for $\Bb_1,\Bb_2$ the boundary matrices of the two graphs. A spectral guarantee, is not available, to our knowledge, for the case where the error of the underlying approximate MM is in terms of the Euclidean norm. Therefore, the techniques of \cite{SS11} on sampling according to effective resistances does not apply. Definition \ref{add_sp_def} on the other hand quantifies the approximation error we get for Laplacians of such intersection graphs.

\section{Experiment}

We compared s.w.r. according to $\{p_e\}_{e\in E}$ (via $\crmm$) which is already known through $w$, and $\{r_e\}_{e\in E}$ (ER); which requires $O(mn^2)$ operations to calculate. Even though our main benefit is algorithmic, empirically our approach performs just as well; in terms of the error characterization \eqref{isot_cond}. We considered the barbell graph on $n=2713$ vertices, and assigned weights to each of the $m=7864$ edges randomly from $\N_{100}$. We sparsified the graph for $r=3500+500\nu$; for each $\nu\in\N_{13}$. In Figure \ref{adj_fig} we present the adjacency matrices of $G$ and $\Gt$, to distinguish the difference of $G$ and $\Gt$ for $r=4000$. In Figures \ref{prec_fig},\ref{error_fig_char}, we show the sparsification rate and error for each $r$.

\begin{figure}[h]
  \centering
    \includegraphics[scale=.35]{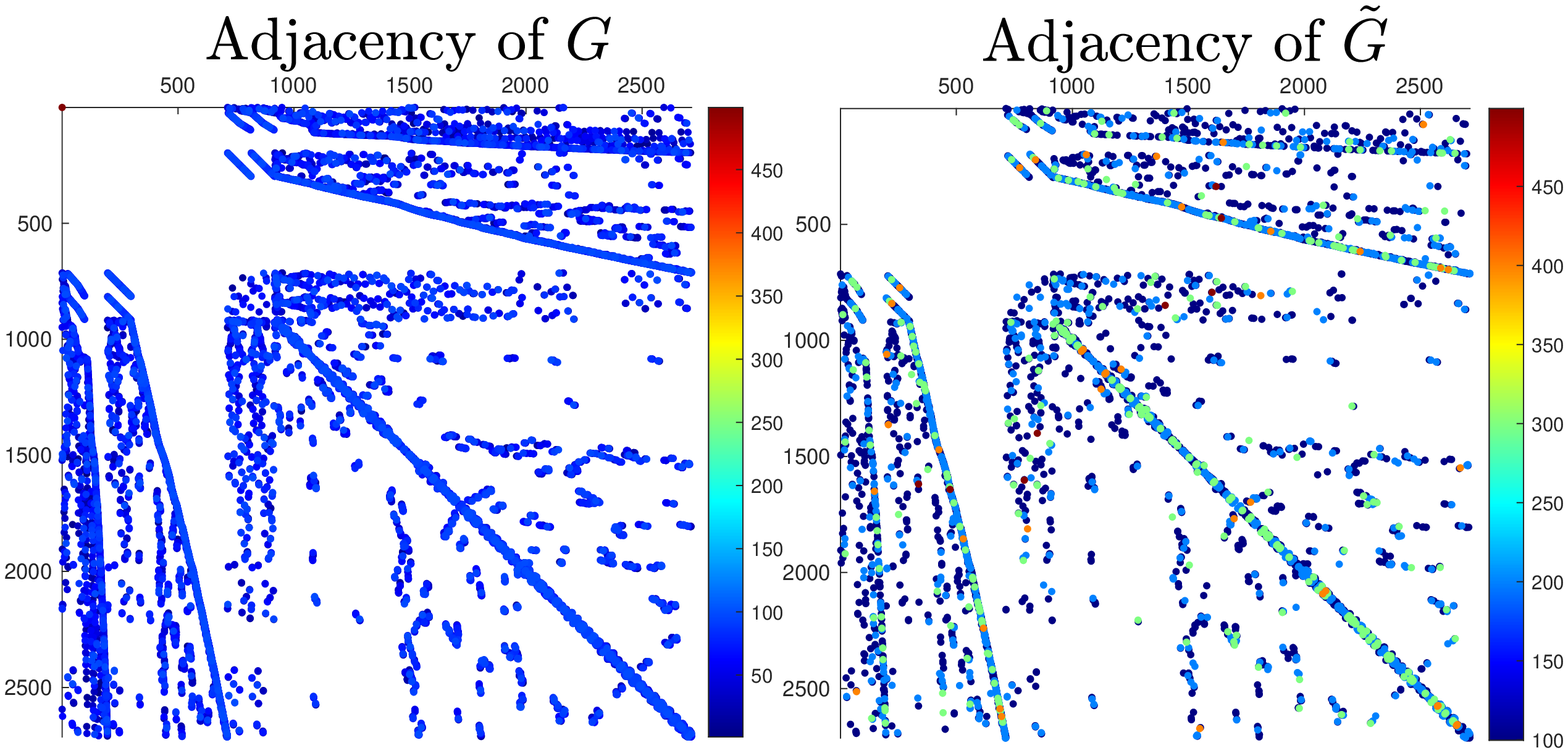}
    \caption{Adjacency matrices of $G$ and $\Gt$, for $r=4000$.}
  \label{adj_fig}
\end{figure}

\begin{figure}[h]
    \centering
    \begin{minipage}{0.5\textwidth}
        \centering
        \includegraphics[scale=.32]{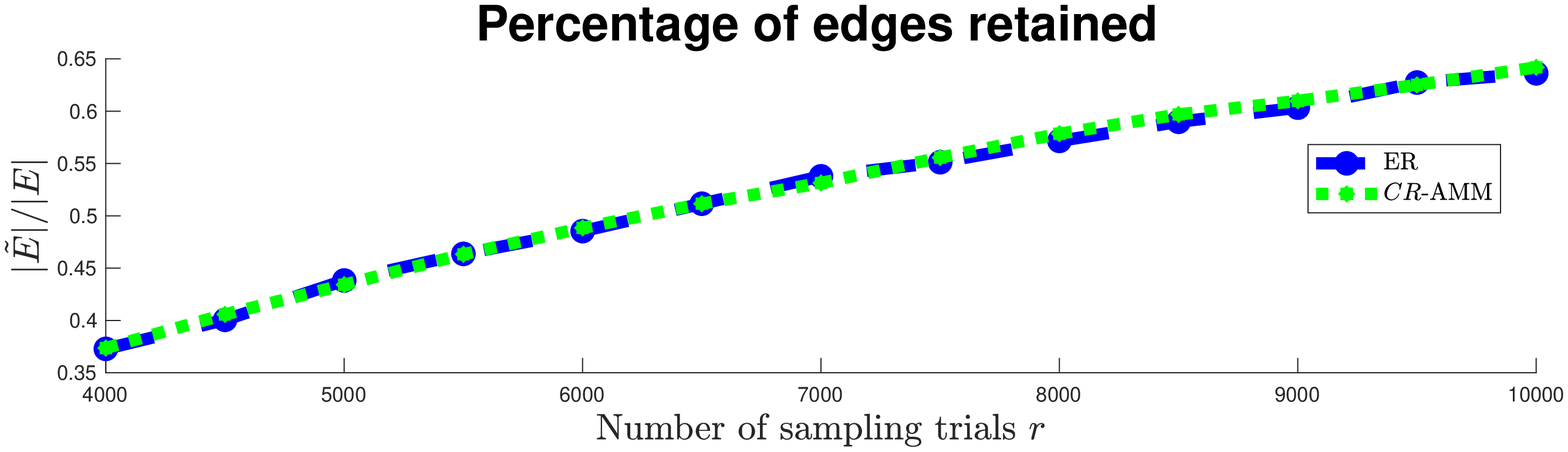}
        \caption{Percentage of retained edges, after sparsification.}
        \label{prec_fig}
    \end{minipage}\hfill
    \begin{minipage}{0.5\textwidth}
        \centering
        \includegraphics[scale=.32]{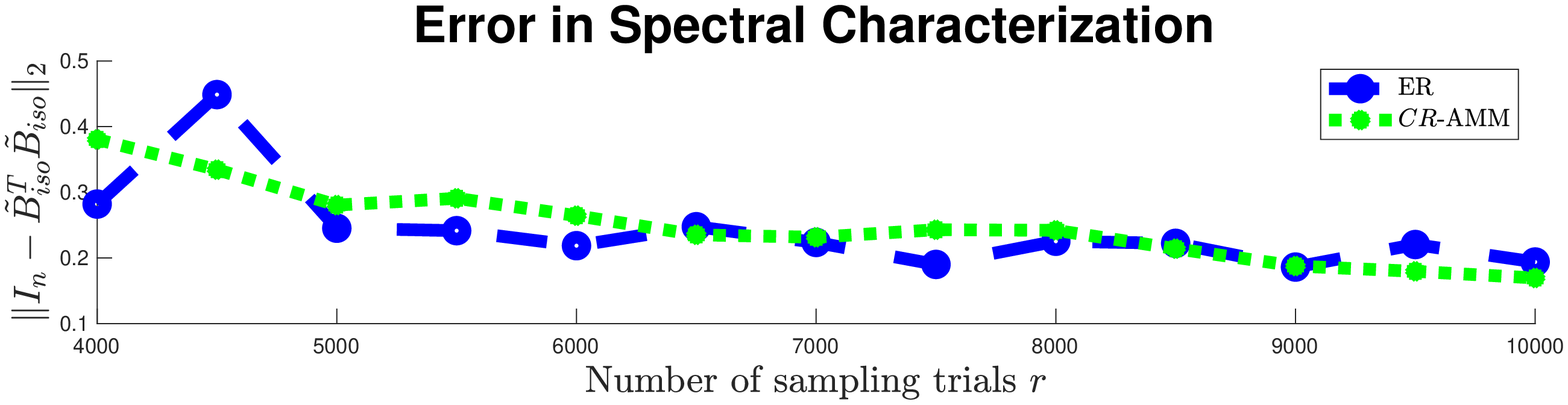}
        \caption{Error in terms of \eqref{isot_cond}, for varying $r$.}
  \label{error_fig_char}

    \end{minipage}
\end{figure}

\section{Future Directions}

In this paper, we proposed a graph sparsifier that approximates Laplacian through the use of $\crmm$; a sampling with replacement technique, adapted from RandNLA.

Applications of the proposed method to spectral clustering through \textit{block} sampling \cite{CPH20c,NL21} would be worthwhile future work. Specifically, cliques of a given graph may be determined by approximating their Laplacians. The proposed computationally efficient spectral approximation may permit the identification of highly connected vertices without the need to traverse through the entire graph.




\balance
\bibliographystyle{unsrt}
\bibliography{refs.bib}

\end{document}